\documentclass[11pt]{article}
\usepackage[T1]{fontenc}
\usepackage{lmodern}
\usepackage[a4paper,left=23mm,right=23mm,top=22mm,bottom=24mm]{geometry}
\usepackage{amsmath,amssymb,amsthm,mathtools,bm}
\usepackage{booktabs,array,tabularx,longtable,enumitem}
\usepackage{algorithm}
\usepackage{algpseudocode}
\floatname{algorithm}{Algorithm}
\algrenewcommand\algorithmicrequire{\textbf{Input:}}
\algrenewcommand\algorithmicensure{\textbf{Output:}}
\usepackage{microtype}
\usepackage[hidelinks]{hyperref}
\usepackage{bookmark}
\hypersetup{
  pdftitle={Exact Values, Extremal Classifications, and Sum-of-Squares Reductions for Second-Order Zarankiewicz Numbers},
  pdfauthor={Yi Xu and Xihong Yan}
}
\setlist{nosep,leftmargin=2em}

\theoremstyle{plain}
\newtheorem{theorem}{Theorem}[section]
\newtheorem{lemma}[theorem]{Lemma}
\newtheorem{proposition}[theorem]{Proposition}

\newtheorem{conjecture}[theorem]{Conjecture}
\theoremstyle{definition}
\newtheorem{definition}[theorem]{Definition}
\newtheorem{example}[theorem]{Example}
\theoremstyle{remark}

\newtheorem{problem}[theorem]{Problem}

\newcommand{\SOS}{\operatorname{SOS}}
\newcommand{\BSR}{\operatorname{BSR}}
\newcommand{\rank}{\operatorname{rank}}
\newcommand{\tr}{\operatorname{tr}}
\newcommand{\supp}{\operatorname{supp}}
\newcommand{\zRL}{z_{\mathrm{RL}}}
\newcommand{\zSL}{z_{\mathrm{SL}}}
\newcommand{\zwL}{z_{\mathrm{wL}}}
\newcommand{\one}{\bullet}
\newcommand{\hole}{\circ}
\newcommand{\Aut}{\operatorname{Aut}}
\newcommand{\norm}[1]{\lVert#1\rVert}
\newcommand{\abs}[1]{\lvert#1\rvert}

\title{Exact Values, Extremal Classifications, and Sum-of-Squares Reductions for Second-Order Zarankiewicz Numbers}
\author{%
  Yi Xu$^{1}$ \qquad Xihong Yan$^{2,*}$\\[0.8em]
  \small $^{1}$Mathematics Department, Southeast University,\\[-0.1em]
  \small 2 Sipailou, Nanjing, Jiangsu Province 210096, P. R. China\\[0.35em]
  \small $^{2}$Key Laboratory of Engineering and Computational Science,\\[-0.1em]
  \small Taiyuan Normal University, Jinzhong 030619, Shanxi, China\\[0.55em]
  \small $^{*}$Corresponding author. E-mail: \href{mailto:yanxihong@tynu.edu.cn}{yanxihong@tynu.edu.cn}\\[-0.1em]
  \small Contributing author. E-mail: \href{mailto:yi.xu1983@hotmail.com}{yi.xu1983@hotmail.com}
}
\date{}

\begin{document}
\maketitle

\begin{abstract}
There is a natural connection between the sum-of-squares (SOS) rank problem for biquadratic forms and the Zarankiewicz extremal problem for $C_4$-free bipartite graphs. The classical Zarankiewicz number $z(m,n)$ controls the bipartite skeleton associated with monomial squares. Allowing two cells to form a single bilinear square leads to augmented Zarankiewicz configurations and to the second-order Zarankiewicz number $z_2(m,n)$. Unlike $\zRL$ and $\zSL$, which are defined through specific recursive sufficient conditions, $z_2$ maximizes over all irreducible displayed SOS decompositions without imposing $(RW3^+)$. Hence, to prove $z_2(m,n)\le R$, one must prove that every simple limited configuration with more displayed squares is reducible; failure of a sufficient condition cannot serve as a counterargument.

We study this distinction for several of the smallest unresolved parameter values. We prove
\[
z_2(4,4)=10,\qquad z_2(7,4)=19,\qquad z_2(8,4)=21,\qquad z_2(5,5)=17,
\]
and obtain $z_2=\zSL=\zRL$ in all of these cases. We also show that the extremal irreducible $6\times4$ configurations with 16 displayed squares form a single isomorphism class under row and column relabeling, whereas the extremal irreducible $7\times4$ configurations with 19 displayed squares form exactly three isomorphism classes. The four-column results form a structural chain: classify lower-order extremal configurations first, then use hereditary irreducibility under deletion of complete squares to constrain the next order. Whenever a finite exhaustive step is needed, we describe candidate pruning, a necessary compatibility graph, clique enumeration, and orbit reduction, and we attach a verifiable reducibility or irreducibility proof to each remaining orbit. For $5\times5$, there are two ordinary extremal skeletons; finite exclusion leaves only two highly symmetric 18-square candidates. They define the same ten-square polynomial, which admits an explicit nine-square representation, yielding the upper bound for $z_2(5,5)$. We conclude with deletion inequalities, global inheritance of local shorter SOS representations, and possible extensions of low-rank Gram methods.
\end{abstract}


\medskip
\noindent\textbf{Funding.} This author's work was supported by Natural Science Foundation of China (No. 12671367). This work was partially supported by the Jiangsu Provincial Scientific Research Center of Applied Mathematics [Grant No. BK20233002].


\section{Introduction}

Sum-of-squares representations of biquadratic forms have both algebraic and combinatorial structure. Let
\[
P(x,y)=\sum_{i,k=1}^{m}\sum_{j,\ell=1}^{n}a_{ijk\ell}x_i x_k y_jy_\ell
\]
be an $m\times n$ biquadratic form. If
\[
P(x,y)=f_1(x,y)^2+\cdots+f_r(x,y)^2,
\]
where the $f_t$ are real bilinear forms, then $P$ is an SOS biquadratic form. The minimum number of squares in such a representation is denoted by $\SOS(P)$. Maximizing over all $m\times n$ SOS biquadratic forms gives the biquadratic SOS rank $\BSR(m,n)$. Understanding $\BSR(m,n)$, and identifying structures that force a given SOS representation to be non-shortenable, are basic problems in this area \cite{CLR,QCXrank}.

This algebraic problem naturally leads to the classical Zarankiewicz problem. Associate the cell $(i,j)$ with the bilinear monomial $x_i y_j$, and for a bipartite graph $G_1=([m],[n],E_1)$ consider
\[
\sum_{(i,j)\in E_1}(x_i y_j)^2.
\]
If the four corners determined by two rows and two columns all lie in $E_1$, write the four corresponding monomials as $a,b,c,d$. Then $ad=bc$, and hence
\[
a^2+b^2+c^2+d^2=(a+d)^2+(b-c)^2.
\]
Thus a single $C_4$ immediately compresses four displayed squares into two. In other words, irreducibility of the displayed decomposition forces the ordinary one-edge graph to be $C_4$-free, so the classical Zarankiewicz number $z(m,n)$ becomes the first combinatorial upper bound for the SOS problem. Classical Zarankiewicz theory begins with the double-counting method of K\H{o}v\'ari--S\'os--Tur\'an and includes exact formulas and asymptotic results in several fixed-width and unbalanced regimes \cite{KST,Culik,Reiman}.

Restricting to monomial squares is still too conservative. If two distinct cells $(i,j)$ and $(k,\ell)$ form the square
\[
(x_i y_j+x_k y_\ell)^2,
\]
then one square may occupy two cells. This observation motivates the augmented bipartite graph framework of Qi, Cui, Xu and collaborators: one-edges represent monomial squares, two-edges represent squares of two-term bilinear forms, and one asks how many two-edges can be added on top of an extremal one-edge skeleton $|E_1|=z(m,n)$ while preserving irreducibility \cite{QCXaug,QCXweak}. Subsequent work introduced several checkable sufficient conditions---weak, recursive-line, and signed---leading to the parameters $\zwL(m,n)$, $\zRL(m,n)$, and $\zSL(m,n)$ \cite{LQ,QCL}.

Among these parameters, $z_2(m,n)$ plays a different role. The quantities $\zRL$ and $\zSL$ are defined by particular sufficient conditions: satisfying the condition certifies that the displayed SOS is irreducible. In contrast, $z_2$ maximizes over all irreducible limited configurations, whether or not a prescribed rule recognizes them. Therefore
\[
 z_2(m,n)\ge \zSL(m,n)\ge \zRL(m,n),
\]
but failure of $(RW3^+)$ does not imply reducibility. This distinction makes upper bounds for $z_2$ fundamentally harder than finding a large construction for $\zRL$: to prove $z_2(m,n)\le R$, every larger candidate must be ruled out by a genuine reducibility proof, such as an explicit shorter SOS representation, a nonzero product identity inducing a Gram-rank reduction, or a positive semidefinite Gram representation of strictly smaller rank.

We focus on the gap between checkable sufficient conditions and irreducibility without an imposed $(RW3^+)$ condition. In three columns it is known that
\[
z_2(m,3)=\zRL(m,3)=2m\qquad(m\ge3),
\]
For four columns the ordinary Zarankiewicz skeleton is already highly rigid when $m\ge6$: $z(m,4)=m+6$, and every extremal one-edge graph consists of six degree-two rows and $m-6$ degree-one rows. Thus four columns provide the first nontrivial testing ground in which the ordinary skeleton is explicit but the relation between $z_2$ and $\zRL$ remains delicate. The $5\times5$ case, by contrast, is one of the smallest balanced square cases in which the ordinary extremal skeleton is no longer unique, and therefore tests whether the methods developed for four columns persist beyond that rigid setting.

We use two complementary methods. The first is structural recursion: completely classify extremal irreducible configurations at a smaller size, use hereditary irreducibility under deletion of complete squares to constrain the next size, and then eliminate impossible extensions by strip bounds, product identities, and explicit shorter SOS representations. The second is local algebraic reduction: first reduce the finite candidate set to a small collection of highly symmetric exceptions, and then construct a lower-rank positive semidefinite Gram representation or an explicit shorter SOS. Computation is used only for finite enumeration, organization into isomorphism orbits, and exact algebraic verification. Every discarded candidate must have an explicit mathematical reason; a failed search or failure of a recursive sufficient condition is never used as a proof of reducibility.

Our main results are as follows:
\begin{enumerate}
\item We prove $z_2(4,4)=\zSL(4,4)=\zRL(4,4)=10$ and show that the extremal irreducible 10-square configuration has a unique isomorphism class;
\item Building on the known value $z_2(6,4)=16$, we classify all extremal irreducible $6\times4$ configurations with 16 displayed squares and obtain a single isomorphism class;
\item We use the six-row classification to prove $z_2(7,4)=19$, and then completely classify the 19-square extremal configurations, obtaining exactly three isomorphism classes;
\item Using the exact seven-row value and the three extremal templates, together with finite enumeration and exact Gram-rank reduction, we prove $z_2(8,4)=21$;
\item For the two ordinary extremal skeletons on $5\times5$, finite exclusion leaves two highly symmetric 18-square candidates. They define the same ten-square polynomial, which has an explicit nine-square representation, and hence $z_2(5,5)=17$;
\item We derive deletion inequalities and a global inheritance principle for local shorter SOS representations, and discuss their significance for general $m\times4$ and fixed-width problems.
\end{enumerate}

The exact four-column formula for $m\ge15$,
\[
z_2(m,4)=\zRL(m,4)=\left\lfloor\frac{5m+6}{2}\right\rfloor
\]
was established by Chen--Chen \cite{CC}; we do not claim it here as a new result. Our emphasis is on closing several small-size upper bounds for $z_2$, classifying extremal configurations, and developing structural methods that can be continued toward the remaining finite cases.

The paper is organized according to proof dependencies. Section~\ref{sec:prelim} introduces the notation, limited configurations, irreducibility, three reduction lemmas, and the finite exclusion procedure for a fixed ordinary skeleton. Section~\ref{sec:44} gives a complete treatment of $4\times4$ and shows how the equality $z_2(4,4)=10$ follows from a lower-bound construction and an upper-bound exclusion. Section~\ref{sec:fourchain} first classifies the 16-square extremal $6\times4$ configurations, then proves $z_2(7,4)=19$ and classifies all 19-square irreducible configurations; that classification then becomes an input to the proof of $z_2(8,4)=21$. Section~\ref{sec:55} treats $5\times5$ via a different pattern: finite exclusion followed by Gram compression. The finite computations follow Algorithm~\ref{alg:finite} and the five types of reducibility proofs described in Section~\ref{subsec:certs}; the tables record only the finite ranges that must be exhausted and the sizes after pruning, while each accepted conclusion remains backed by a verifiable reducibility or irreducibility proof.

\section{Preliminaries, notation, and basic tools}\label{sec:prelim}

\subsection{Biquadratic forms, SOS rank, and Gram representations}

Let $x=(x_1,\ldots,x_m)$ and $y=(y_1,\ldots,y_n)$. An $m\times n$ biquadratic form is
\[
P(x,y)=\sum_{i,k=1}^{m}\sum_{j,\ell=1}^{n}a_{ijk\ell}x_i x_k y_jy_\ell.
\]
If there exist real bilinear forms $f_1,\ldots,f_r$ such that
\[
P=\sum_{t=1}^{r}f_t^2,
\]
then $P$ is called an SOS biquadratic form, and we define
\[
\SOS(P)=\min\left\{r:P=\sum_{t=1}^{r}f_t^2\right\}.
\]
We further define
\[
\BSR(m,n)=\max\{\SOS(P):P\text{ is an }m\times n\text{ SOS biquadratic form}\}.
\]

Arrange all monomials $u_{ij}=x_i y_j$ into a column vector $u$. If there exists a symmetric positive semidefinite matrix $Q\succeq0$ such that
\[
P=u^{\mathsf T}Qu,
\]
then $Q$ is a Gram matrix of $P$. If $\rank Q=r$, a factorization $Q=V^{\mathsf T}V$ yields a representation by $r$ bilinear squares, and hence
\[
\SOS(P)\le \rank Q.
\]
All low-rank Gram arguments below are applications of this elementary fact.

\subsection{Augmented bipartite configurations}

Let $[m]=\{1,\ldots,m\}$, and identify the cell $(i,j)\in[m]\times[n]$ with the monomial $u_{ij}=x_i y_j$. In specific finite configurations we sometimes relabel rows and columns as $0,1,\ldots$ or $A,B,\ldots$ to exploit symmetry. The corresponding symbols $u_{ij}$, $A3$, and so on refer to the relabeled cell monomials. For example, $A3$ denotes the monomial in row $A$ and column~3.

\begin{definition}[One-edges and two-edges]
A \emph{one-edge} is a single cell $(i,j)$. A \emph{two-edge} is an unordered pair of two distinct cells
\[
e=\{(i,j),(k,\ell)\}.
\]
If $i=k$, $e$ is row-degenerate; if $j=\ell$, it is column-degenerate; and if $i\ne k$ and $j\ne\ell$, it is nondegenerate. If both diagonals of a genuine rectangle are selected as two-edges, they are called complementary two-edges.
\end{definition}

For a one-edge $e=(i,j)$ set $\supp(e)=\{(i,j)\}$; for a two-edge, $\supp(e)$ is the set of its two cells. We write an augmented configuration as
\[
G=([m],[n],E_1\cup E_2),
\]
where $E_1$ is the set of one-edges and $E_2$ is the set of two-edges.

\begin{definition}[Simplicity]
The configuration $G$ satisfies the simplicity condition $(S)$ if the supports of distinct selected edges are disjoint.
\end{definition}

The displayed SOS associated with $G$ is
\begin{equation}\label{eq:PG}
P_G(x,y)=\sum_{(i,j)\in E_1}u_{ij}^2+
\sum_{\{(i,j),(k,\ell)\}\in E_2}(u_{ij}+u_{k\ell})^2.
\end{equation}
Write
\[
R(G)=|E_1|+|E_2|
\]
for the number of displayed squares, and let $H(G)$ be the number of cells not occupied by any selected edge. A two-edge occupies two cells but contributes only one square. Simplicity gives
\begin{equation}\label{eq:cellcount}
|E_1|+2|E_2|+H=mn,
\qquad
2R(G)=mn+|E_1|-H.
\end{equation}

\begin{definition}[Irreducibility of the displayed decomposition]
If $\SOS(P_G)=R(G)$, then the displayed decomposition in \eqref{eq:PG} is called irreducible; in that case we also call the configuration $G$ irreducible.
\end{definition}

In grid diagrams, $\one$ denotes a one-edge, $\hole$ denotes a hole, and two occurrences of the same letter represent one two-edge.

\begin{definition}[Ordinary extremal skeleton and free cells]\label{def:skeleton}
Fix $m,n$. If a $C_4$-free $m\times n$ bipartite graph has $z(m,n)$ edges, we call its edge set $E_1$ an ordinary extremal skeleton. Cells not in $E_1$ are called free cells relative to this skeleton. Two-edges in a limited configuration may use only free cells.
\end{definition}

\begin{definition}[Row-column isomorphism]\label{def:iso}
Two augmented configurations are row-column isomorphic if a permutation of the rows and a permutation of the columns map the one-edge set and the two-edge set of one configuration to those of the other. Every isomorphism class below is understood in this sense.
\end{definition}

\subsection{Zarankiewicz numbers and second-order parameters}

The classical Zarankiewicz number $z(m,n)$ is the maximum number of edges in a $C_4$-free $m\times n$ bipartite graph.

\begin{definition}[Limited configuration]
If $G$ satisfies simplicity $(S)$, the one-edge graph $E_1$ is $C_4$-free, and
\[
|E_1|=z(m,n),
\]
then $G$ is called a \emph{limited} augmented configuration. The term ``limited configuration'' will always have this meaning.
\end{definition}

\begin{definition}[Second-order Zarankiewicz number]
Define
\[
z_2(m,n)=\max\left\{R(G):
\begin{array}{l}
G\text{ is a simple limited configuration, and}\\
\text{the displayed decomposition of }P_G\text{ is irreducible}
\end{array}
\right\}.
\]
\end{definition}

To obtain checkable sufficient conditions for irreducibility, L\"ofberg--Qi \cite{LQ} introduced the strengthened recursive rectangle condition $(RW3^+)$ and the signed condition $(RW3^{\pm})$. The parameter $\zRL(m,n)$ is the maximum number of displayed squares among simple limited configurations satisfying $(RW3^+)$; $\zSL(m,n)$ is defined analogously using the signed condition $(RW3^{\pm})$; and the weak parameter $\zwL(m,n)$ is defined by an earlier sufficient condition. We only use the established hierarchy
\begin{equation}\label{eq:hierarchy}
\BSR(m,n)\ge z_2(m,n)\ge \zSL(m,n)\ge \zRL(m,n)\ge \zwL(m,n)\ge z(m,n).
\end{equation}

We next record the form of $(RW3^+)$ used in this paper. Consider any SOS representation
\[
P_G=\sum_{t=1}^{s}g_t^2,
\]
and let $v_p\in\mathbb R^s$ be the coefficient vector of the cell monomial corresponding to $p$ across the bilinear forms $g_t$. Let $\delta(p,q)=1$ if and only if $\{p,q\}\in E_2$, and let it be zero otherwise. Coefficient comparison gives
\begin{align}
\norm{v_p}^2&=1 &&(p\text{ occupied}),\label{eq:norm}\\
v_p&=0 &&(p\text{ unoccupied}),\label{eq:hole}\\
\langle v_p,v_q\rangle&=\delta(p,q) &&(p,q\text{ in a common row or column}),\label{eq:line}\\
\langle v_p,v_q\rangle+\langle v_r,v_s\rangle
&=\delta(p,q)+\delta(r,s) &&(\{p,q\},\{r,s\}\text{ are the two diagonals of one rectangle}).\label{eq:rect}
\end{align}
These equations induce an identification relation $\sim$ and a certified orthogonality relation $\perp_R$. We use the following four closure rules.

\begin{definition}[Generating rules for $(RW3^+)$]\label{def:rw3rules}
Fix a simple configuration $G$, and let the occupied-cell set be
\[
\Omega=E_1\cup\bigcup_{e\in E_2}e.
\]
For distinct cells $p,q$, define the prescribed value
\[
\delta(p,q)=\begin{cases}
1,&\{p,q\}\in E_2,\\
0,&\text{otherwise}.
\end{cases}
\]
On $\Omega$, construct the least equivalence relation $\sim$ and the least symmetric orthogonality relation $\perp_R$ closed under the following finite rules.
\begin{enumerate}
\item \textbf{Line rule.} If occupied cells $p,q\in\Omega$ lie in a common row or column, certify their prescribed value from \eqref{eq:line}: if $\delta(p,q)=1$, add $p\sim q$; if $\delta(p,q)=0$, add $p\perp_R q$.
\item \textbf{Saturation rule.} If
\[
p\sim p',\qquad q\sim q',\qquad p'\perp_R q',
\]
then add $p\perp_R q$. Thus a certified orthogonality relation may be transported along identified classes, but this rule does not assign a value to an otherwise undetermined inner product.
\item \textbf{Rectangle transfer rule.} Let $\{p,q\}$ and $\{r,s\}$ be the two diagonals of a genuine rectangle. Suppose the companion diagonal $\{r,s\}$ has already been certified to take its prescribed value: if $\delta(r,s)=0$, then at least one of $r,s$ is unoccupied, or $r,s\in\Omega$ and $r\perp_R s$; if $\delta(r,s)=1$, then $r,s\in\Omega$ and $r\sim s$. Equation~\eqref{eq:rect} then certifies the prescribed value on the target diagonal. If $p,q\in\Omega$ and $\delta(p,q)=0$, add $p\perp_R q$; if $\delta(p,q)=1$, then necessarily $p,q\in\Omega$, and we add $p\sim q$.
\item \textbf{Complementary-pair rule.} If both diagonals $\{p,q\}$ and $\{r,s\}$ of a genuine rectangle are selected two-edges, then \eqref{eq:rect} gives
\[
\langle v_p,v_q\rangle+\langle v_r,v_s\rangle=2.
\]
By \eqref{eq:norm} and the Cauchy--Schwarz inequality, each inner product is at most one. Hence both equal one, so we add $p\sim q$ and $r\sim s$ simultaneously.
\end{enumerate}
The closure is the least fixed point of these rules; only relations obtained by finitely many rule applications are used.
\end{definition}

\begin{definition}[$(RW3^+)$-admissibility]\label{def:rw3}
A simple configuration $G$ is called $(RW3^+)$-admissible if the closure of Definition~\ref{def:rw3rules} has the following properties: the two halves of every selected two-edge are identified; distinct selected edges lie in distinct equivalence classes; and every pair of distinct selected edges has orthogonal representatives.
\end{definition}

\begin{lemma}[Soundness of $(RW3^+)$]\label{lem:rw3}
If $G$ is $(RW3^+)$-admissible, then the displayed decomposition $P_G$ is irreducible.
\end{lemma}
\begin{proof}
The identifications and orthogonality relations generated by the closure hold for the coefficient vectors in every SOS representation. Hence each selected edge contributes a unit representative vector, and representatives of distinct selected edges are mutually orthogonal. Therefore every SOS representation requires at least $R(G)$ squares.
\end{proof}

Whenever we say that $(RW3^+)$ is fully verified or that its closure succeeds, we mean that $G$ has been checked to be $(RW3^+)$-admissible in the sense of Definition~\ref{def:rw3}. This is a sufficient condition only; failure does not by itself prove reducibility. When the full closure is not printed in the text, it is included in Supplementary Data~$\mathrm{S1}$, where an independent checker replays all identifications and orthogonality steps using the four rules of Definition~\ref{def:rw3rules}.

\subsection{Three basic reduction lemmas}

\begin{lemma}[Deletion of complete displayed squares]\label{lem:delete}
If a displayed SOS decomposition is irreducible, then deleting any collection of complete displayed squares leaves an irreducible displayed decomposition.
\end{lemma}
\begin{proof}
If the remainder had a shorter representation, adding back the deleted squares would shorten the original decomposition.
\end{proof}

\begin{lemma}[Two-row or two-column strip bound]\label{lem:strip}
If a complete displayed sub-sum is supported on two rows and $q$ columns, or on $q$ rows and two columns, and contains more than $q+1$ squares, then that sub-sum is reducible.
\end{lemma}
\begin{proof}
Apply the known bounds $\BSR(q,2)\le q+1$ and $\BSR(2,q)\le q+1$ \cite{QCXrank,LQ}, and then add back the remaining displayed squares. The sub-sum itself need not satisfy the limited condition.
\end{proof}

\begin{lemma}[Product relation implies rank reduction]\label{lem:product}
Let the displayed bilinear forms be $f_1,\ldots,f_R$. Suppose there are coefficients, not all zero, such that
\[
\sum_{a<b}c_{ab}f_af_b=0,
\]
and, after collecting identical formal products, the coefficient vector is nonzero. Then $\sum_{a=1}^{R}f_a^2$ is reducible.
\end{lemma}
\begin{proof}
Define a symmetric matrix $H$ by $H_{aa}=0$ and $H_{ab}=c_{ab}/2$. Then $H\ne0$, $\tr H=0$, and $f^{\mathsf T}Hf=0$. Let $\lambda=\lambda_{\max}(H)>0$. Then
\[
Q=I-H/\lambda\succeq0,
\qquad
\rank Q\le R-1,
\qquad
f^{\mathsf T}Qf=f^{\mathsf T}f.
\]
A positive semidefinite factorization of $Q$ yields a shorter SOS representation.
\end{proof}

We also repeatedly use the following two local identities from \cite{LQ}, with the indicated row and column labels pairwise distinct:
\begin{align}
&(xb+za)^2+(ya+zd)^2+(xd)^2+(yb)^2+(yd)^2\notag\\
&\qquad=(xb+yd+za)^2+(xd-yb)^2+(ya)^2+(zd)^2,\label{eq:S}\\[1mm]
&(xa)^2+(zc)^2+(ya)^2+(yb)^2+(zb)^2+(xc+zd)^2+(xd+za)^2\notag\\
&\qquad=(xc)^2+(ya+zb)^2+(yb-za)^2+(xd+zc)^2+(xa+zd)^2.\label{eq:E}
\end{align}
They reduce five squares to four and seven squares to five, respectively.

\subsection{Finite exclusion for a fixed ordinary skeleton}

Once the ordinary extremal skeleton $E_1$ is fixed, let $F$ denote the set of free cells. A candidate two-edge is an unordered pair of cells in $F$, so there are $\binom{|F|}{2}$ single candidates. If the target number of displayed squares is $R$, we must choose
\[
k=R-|E_1|
\]
pairwise support-disjoint two-edges. With $H=|F|-2k$ holes, the number of labeled families before pruning is
\begin{equation}\label{eq:rawcount}
M(|F|,k)=\frac{|F|!}{H!\,2^k\,k!}.
\end{equation}
For seven rows and 19 displayed squares we have $(|F|,k,H)=(15,6,3)$ and hence $M(15,6)=4\,729\,725$; for eight rows and 22 squares we have $(18,8,2)$ and $M(18,8)=310\,134\,825$ for each ordinary skeleton type; for five rows and 18 squares we have $(13,6,1)$ and $M(13,6)=270\,270$. These counts only show that the search space is finite. Testing every family directly for SOS rank would be opaque and unnecessary. We therefore separate the exhaustion mechanism from the mathematical reason why each rejected object is reducible.

\begin{algorithm}[htbp]
\caption{Finite exclusion for a fixed ordinary skeleton}\label{alg:finite}
\begin{algorithmic}[1]
\Require ordinary extremal skeleton $E_1$, free-cell set $F$, target number of squares $R$
\Ensure Let $k=R-|E_1|$. If every orbit is resolved, output the decision and proof for every orbit and list all irreducible orbit representatives; otherwise return ``unresolved orbit exists and do not claim a complete classification.
\State Enumerate all unordered pairs in $F$; call the set of single candidates $\mathcal C$
\ForAll{candidates $e\in\mathcal C$}
  \If{$e$ together with $E_1$ already yields a strip overload, a product relation from Lemma~\ref{lem:product}, or an instance of \eqref{eq:S}--\eqref{eq:E}}
    \State Delete $e$ permanently
  \EndIf
\EndFor
\State For pairs of retained candidates: declare them incompatible if their supports overlap, or if together with $E_1$ they already yield a verified local reduction
\State Build the necessary compatibility graph $G_M$ on retained candidates; join two vertices only when their supports are disjoint and no verified two-candidate reduction applies
\State Enumerate all $k$-cliques of $G_M$ exactly
\ForAll{$k$-cliques $K$}
  \If{$K$ contains a strip overload, product relation, or local identity involving at least three candidates}
    \State Delete $K$
  \EndIf
\EndFor
\State Compute the row-column automorphism group $\Aut(E_1)$ preserving $E_1$, and partition the remaining labeled families into orbits
\ForAll{orbit representatives $K$}
  \If{there is an explicit reducibility proof of a type listed in Section~\ref{subsec:certs}}
    \State Mark the orbit reducible
  \ElsIf{$K$ is $(RW3^+)$-admissible by Definition~\ref{def:rw3}}
    \State Mark the orbit irreducible
  \Else
    \State Mark as unresolved; draw no reducibility or irreducibility conclusion
  \EndIf
\EndFor
\If{an unresolved orbit exists}
  \State Return ``unresolved orbit exists; do not claim a complete finite classification
\Else
  \State Output all orbit decisions and proofs, and list all irreducible orbit representatives
\EndIf
\end{algorithmic}
\end{algorithm}

The compatibility graph provides only necessary conditions. Pairwise local compatibility does not imply that a family of six or eight two-edges is irreducible. The two final 18-square families in the $5\times5$ case are a useful warning: they pass many low-order checks, yet a ten-square sub-sum can still be compressed globally to nine squares. The finite parts of the $7\times4$, $8\times4$, and $5\times5$ arguments all follow Algorithm~\ref{alg:finite}.

\subsection{Positive cases, negative cases, and five types of reducibility proof}\label{subsec:certs}

The logical treatment of positive and negative cases must be kept separate. If a configuration is $(RW3^+)$-admissible, Lemma~\ref{lem:rw3} proves that it is irreducible. Conversely, failure of $(RW3^+)$ only means that this sufficient condition did not succeed; it does not imply reducibility. Accordingly, positive cases are accepted only after a full $(RW3^+)$ verification or another proof of irreducibility, whereas negative cases are accepted only after a shorter SOS, a nonzero product relation, or a low-rank positive semidefinite Gram representation. We never use ``$(RW3^+)$ fails or ``the program found no positive example as a proof of reducibility.

The negative cases in this paper use only the following five types of proof.

{\raggedright
\begin{enumerate}[leftmargin=2em,itemsep=0.45em]
\item \textbf{Strip overload.}\\
A complete displayed sub-sum is supported on two rows and $q$ columns, or on $q$ rows and two columns, and contains more than $q+1$ squares. By Lemma~\ref{lem:strip}, that sub-sum is reducible, hence so is the full displayed decomposition.
\item \textbf{Nonzero product relation in the original displayed space.}\\
For displayed forms $f_1,\ldots,f_R$, find coefficients not all zero such that $\sum_{a<b}c_{ab}f_af_b=0$. Lemma~\ref{lem:product} then gives a positive semidefinite Gram matrix of rank at most $R-1$.
\item \textbf{Explicit shorter integer SOS.}\\
Give bilinear forms $g_1,\ldots,g_s$ with integer coefficients, with $s<R$, and verify by exact expansion that $\sum_t g_t^2=P_G$.
\item \textbf{Product relation after an equal-length rewrite.}\\
First give $R$ bilinear forms $g_1,\ldots,g_R$ with integer coefficients and verify $\sum_t g_t^2=P_G$; then find a nonzero product relation among the $g_t$. This is still an application of Lemma~\ref{lem:product}, but the relation appears only after an equal-length change of displayed forms.
\item \textbf{Low-rank positive semidefinite Gram representation.}\\
Give a symmetric matrix $Q$ in the full cell-monomial space such that $u^{\mathsf T}Qu=P_G$, $Q\succeq0$, and $\rank Q<R$. Factoring $Q=V^{\mathsf T}V$ yields a shorter SOS. We use two exact verification mechanisms: explicit matrices over a real quadratic field, and rational contraction data; see Lemmas~\ref{lem:qgram} and~\ref{lem:contraction}.
\end{enumerate}
\par}

For eight rows we use one additional structural exclusion. After deleting a degree-one row, the seven-row restriction either has more than 19 displayed squares, or has exactly 19 but is not isomorphic to any of the three classes in Theorem~\ref{thm:74classes}. This is not a failure of a recursive test; it is a direct consequence of the established equality $z_2(7,4)=19$ and the complete three-class classification.

Numerical SDP or local search may be used to discover proof objects, but it is not part of the final acceptance criterion. Archived data contain only exact integers, rationals, or entries in $\mathbb Q(\sqrt2)$. Independent verification follows Algorithm~\ref{alg:verify}: candidate indices are regenerated, local supports are checked, polynomials are expanded exactly, the compatibility graph is rebuilt, target cliques are re-enumerated, automorphism orbits are recomputed, and the reducibility or irreducibility proof for every orbit is verified.

\begin{algorithm}[htbp]
\caption{Independent verification of orbit proofs}\label{alg:verify}
\begin{algorithmic}[1]
\Require skeleton $E_1$, free-cell set $F$, target number of squares $R$, and archived exact data for every orbit representative
\Ensure Let $k=R-|E_1|$. Accept the finite classification only if every orbit is resolved; otherwise declare the classification incomplete
\State Independently rebuild the single-candidate exclusions, two-candidate exclusions, and the necessary compatibility graph $G_M$
\State Independently enumerate all target $k$-cliques and construct orbits under $\Aut(E_1)$
\State Require the independently reconstructed orbit set to agree exactly with the archived orbit set
\ForAll{orbit representatives $K$}
  \If{$K$ has a nonzero product relation in the original displayed space}
    \State Accept as reducible by Lemma~\ref{lem:product}
  \ElsIf{the data provide a shorter integer SOS or a product relation after an equal-length rewrite}
    \State Verify $\sum_t g_t^2=P_G$ by exact expansion; in the second case also verify the nonzero product relation
  \ElsIf{the data give an explicit Gram matrix over the quadratic field}
    \State Perform the exact elimination of Lemma~\ref{lem:qgram} and require rank smaller than the displayed length
  \ElsIf{the data give a rational contraction proof}
    \State Use Lemma~\ref{lem:contraction} to verify existence of a positive semidefinite matrix of smaller rank inside the box
  \ElsIf{the data invoke a seven-row restriction}
    \State Delete the specified degree-one row and check the remaining length and the three isomorphism classes of Theorem~\ref{thm:74classes}
  \ElsIf{the data claim $(RW3^+)$-admissibility}
    \State Fully verify the closure according to Definition~\ref{def:rw3}
  \Else
    \State Mark as unresolved; draw no reducibility or irreducibility conclusion
  \EndIf
\EndFor
\If{an unresolved orbit exists}
  \State Return ``unresolved orbit exists; do not claim a complete finite classification
\Else
  \State Accept the finite classification
\EndIf
\end{algorithmic}
\end{algorithm}

\subsection{Verification of exact Gram representations}\label{subsec:gramcheck}

Arrange the monomials of the occupied cells into a column vector $u$. If a symmetric positive semidefinite matrix $Q$ satisfies $u^{\mathsf T}Qu=P_G$ and $\rank Q=r$, then a factorization $Q=V^{\mathsf T}V$ expresses $P_G$ as the sum of squares of the $r$ components of $Vu$. Thus a positive semidefinite Gram matrix of rank smaller than the displayed length is precisely a shorter SOS representation.

Each exact Gram construction first provides rational matrices $B_0,B_1,\ldots,B_d$ and a rational linear map $L$ satisfying
\begin{equation}\label{eq:affinemodel}
u^{\mathsf T}L\Bigl(B_0+\sum_{j=1}^{d}y_jB_j\Bigr)L^{\mathsf T}u
\equiv P_G\qquad\text{for every real parameter vector }y.
\end{equation}
Verification is by exact expansion: $B_0$ produces the target coefficients and each $B_j$ produces the zero polynomial. In most records, $L$ only groups certain cells into a common coordinate. If $Q(y)=B_0+\sum_j y_jB_j\succeq0$ and $\rank Q(y)\le r$, then $LQ(y)L^{\mathsf T}$ is also positive semidefinite of rank at most $r$, yielding a representation by at most $r$ squares.

\begin{lemma}[Explicit Gram representation over a quadratic field]\label{lem:qgram}
Suppose the data give a symmetric matrix $Q$ over $\mathbb Q(\sqrt2)$, with each entry stored as a rational pair $(a,b)$ representing $a+b\sqrt2$. Perform symmetric elimination: when the current pivot is strictly positive, take the Schur complement and increase the rank by one; when the pivot is zero, require the corresponding remaining row to vanish identically. Signs in the quadratic field are decided by rational comparison and, when necessary, by comparing $a^2$ with $2b^2$. If the elimination succeeds and the resulting rank $r$ is smaller than the displayed length, then $P_G$ is reducible.
\end{lemma}
\begin{proof}
This is exact symmetric Gaussian elimination over the real quadratic field. A positive pivot contributes one positive inertia direction; a zero pivot with a zero remaining row contributes a kernel direction; a negative pivot rejects the record. Only exact arithmetic and sign determination in $\mathbb Q(\sqrt2)$ are used, so both positive semidefiniteness and rank are exact conclusions. From $u^{\mathsf T}Qu=P_G$ we obtain $\SOS(P_G)\le r$.
\end{proof}

The remaining exact Gram arguments do not give a closed-form matrix. Instead they prove that an exact low-rank positive semidefinite matrix exists inside a small rational box. After a common row-column permutation, write the affine model $Q(y)$ as
\[
Q(y)=\begin{pmatrix}A(y)&D(y)\\D(y)^{\mathsf T}&E(y)\end{pmatrix},
\qquad A(y)\text{ is }r\times r,
\]
and let $k=N-r$. It is enough that
\begin{equation}\label{eq:Schurzero}
A(y)\succ0,\qquad
S(y):=E(y)-D(y)^{\mathsf T}A(y)^{-1}D(y)=0,
\end{equation}
Then $Q(y)\succeq0$ and $\rank Q(y)=r$. Let $F(y)$ be the vector of independent upper-triangular entries of $S(y)$, giving $e=k(k+1)/2$ equations. The archived data select $e$ active parameters, while the remaining parameters are fixed at rational values.

Below, vector norms are max norms and matrix norms are maximum absolute row sums. Let the active variables be $y\in\mathbb R^e$, with rational center $\bar y$ and radius $\varepsilon>0$. At the center define
\[
A_0=A(\bar y),\quad D_0=D(\bar y),\quad
V=A_0^{-1},\quad W_0=VD_0,\quad
F_0=F(\bar y),\quad J_0=DF(\bar y).
\]
Write each active direction in the same block form as
\[
B_j=\begin{pmatrix}A_j&D_j\\D_j^{\mathsf T}&E_j\end{pmatrix}.
\]
Let $Y$ be a rational $e\times e$ matrix, and define
\begin{align*}
\alpha&=\norm{YF_0}_{\infty},&
\beta&=\norm{I-YJ_0}_{\infty},\\
a&=\norm{V}_{\infty},&w&=\norm{W_0}_{\infty},\\
C_A&=\norm{\sum_j\abs{A_j}}_{\infty},&
C_D&=\norm{\sum_j\abs{D_j}}_{\infty}.
\end{align*}
When $\gamma=a\varepsilon C_A<1$, further define
\begin{align}
\delta&=\frac{a\varepsilon(C_D+C_Aw)}{1-\gamma},\label{eq:deltabound}\\
\Lambda&=\delta\sum_j\bigl(
\norm{D_j^{\mathsf T}}_{\infty}+r\norm{D_j}_{\infty}
+r\norm{A_j}_{\infty}(2w+\delta)\bigr),\label{eq:Lambdabound}\\
\theta&=\beta+\norm{Y}_{\infty}\Lambda.\label{eq:thetabound}
\end{align}
All quantities are obtained by finitely many operations on rational data.

\begin{lemma}[Rational existence proof for a low-rank positive semidefinite matrix]\label{lem:contraction}
If $A_0\succ0$, $\gamma<1$, and
\begin{equation}\label{eq:certificateconditions}
\theta<1,\qquad \alpha+\theta\varepsilon<\varepsilon,
\end{equation}
then there exists $y^*$ in the box $\norm{y-\bar y}_{\infty}\le\varepsilon$ such that $F(y^*)=0$; the corresponding matrix $Q(y^*)$ is positive semidefinite of rank $r$.
\end{lemma}
\begin{proof}
First, throughout the box,
\[
\norm{A(y)-A_0}_{\infty}\le\varepsilon C_A,\qquad
\norm{D(y)-D_0}_{\infty}\le\varepsilon C_D.
\]
Since $\norm{V(A-A_0)}_{\infty}\le\gamma<1$, the Neumann series yields $\norm{A(y)^{-1}}_{\infty}\le a/(1-\gamma)$. Hence $A(y)$ is nonsingular throughout the box. Along the segment from the center to any point of the box, $A(y)$ is continuous and symmetric; starting from the positive definite matrix $A_0$, it cannot lose positive definiteness without acquiring a zero eigenvalue. Therefore $A(y)\succ0$ throughout the box.

Let $W(y)=A(y)^{-1}D(y)$. From
\[
W(y)-W_0=A(y)^{-1}\bigl[(D-D_0)-(A-A_0)W_0\bigr]
\]
we obtain $\norm{W(y)-W_0}_{\infty}\le\delta$. The directional derivative of the Schur complement is
\[
\partial_j S=E_j-D_j^{\mathsf T}W-W^{\mathsf T}D_j+W^{\mathsf T}A_jW.
\]
Using $\norm{W^{\mathsf T}}_{\infty}\le r\norm{W}_{\infty}$ and comparing the derivative at the center with that at an arbitrary point of the box term by term, the change in each direction is at most
\[
\delta\bigl(\norm{D_j^{\mathsf T}}_{\infty}+r\norm{D_j}_{\infty}
+r\norm{A_j}_{\infty}(2w+\delta)\bigr).
\]
Extracting the upper-triangular entries does not increase any absolute value; summing over directions gives $\norm{DF(y)-J_0}_{\infty}\le\Lambda$. Hence the derivative norm of the map $\Phi(y)=y-YF(y)$ is at most $\theta<1$, and
\[
\norm{\Phi(y)-\bar y}_{\infty}
\le\norm{\Phi(\bar y)-\bar y}_{\infty}+\theta\norm{y-\bar y}_{\infty}
\le\alpha+\theta\varepsilon<\varepsilon.
\]
Thus $\Phi$ maps the closed box into itself and is a strict contraction, so iteration from any point converges to the unique fixed point in the box. Moreover $\beta\le\theta<1$, so $YJ_0$ is invertible and hence $Y$ is invertible. The fixed-point condition $YF(y^*)=0$ therefore implies $F(y^*)=0$. Finally, \eqref{eq:Schurzero} gives the claimed rank and positive semidefiniteness.
\end{proof}

Numerical computation is used only to suggest active variables, a center, and an approximate inverse. In the archived record, the center, radius, $Y$, and affine basis are all rational. Verification recomputes the rational elimination of $A_0$, the quantities $F_0,J_0$, the norm bounds, and the inequalities in \eqref{eq:certificateconditions}, and accepts the record only when the strict inequalities hold exactly.

To illustrate that the reducibility proofs are concrete rather than abstract labels, we give representatives of two of the shortest negative proofs. Both arise in the seven-row 19-square classification with the ordinary skeleton fixed and the new one-edge equal to $h1$. As above, $A3$ denotes the monomial in row $A$ and column~3.

\begin{example}[Explicit shorter integer SOS]\label{ex:intsos}
Take the six two-edges
\[
A3+B4,\quad A4+h2,\quad B2+h3,\quad C2+F1,\quad D1+D4,\quad E3+h4.
\]
The corresponding 19-square displayed decomposition equals the following sum of 17 squares:
\begin{align*}
&(A1)^2+(A3+B4-E2)^2+(A4)^2+(B1)^2+(B2+E4)^2+(B3+h2)^2\\
&+(C1)^2+(C2+F1)^2+(C4)^2+(D1+D4)^2+(D2)^2+(D3)^2\\
&+(A2+E3+h4)^2+(F3)^2+(F4)^2+(h1)^2+(h3)^2.
\end{align*}
Exact expansion shows that the two sides have identical coefficients as biquadratic forms, so this 19-square configuration is reducible.
\end{example}

\begin{example}[Product relation after an equal-length rewrite]\label{ex:switch}
Take the six two-edges
\[
A3+D4,\quad A4+h2,\quad B2+h3,\quad B4+F2,\quad C2+E3,\quad C3+h4.
\]
Set
\begin{align*}
g_1&=A1,&
g_2&=A2+h4,&
g_3&=A3+D4,&
g_4&=A4,\\
g_5&=B1,&
g_6&=B2+F4,&
g_7&=B3+h2,&
g_8&=B4,\\
g_9&=C1,&
g_{10}&=C2+E3,&
g_{11}&=C3,&
g_{12}&=C4+h3,\\
g_{13}&=D2,&
g_{14}&=D3,&
g_{15}&=E2,&
g_{16}&=E4,\\
g_{17}&=F2,&
g_{18}&=F3,&
g_{19}&=h1.
\end{align*}
Then $\sum_{t=1}^{19}g_t^2$ equals the original 19-square displayed decomposition, and
\[
g_1g_8-g_4g_5=0.
\]
By Lemma~\ref{lem:product}, the configuration is reducible.
\end{example}

\subsection{Universal cell bound}

By \eqref{eq:cellcount}, if $G$ is limited, then
\[
2R(G)=mn+z(m,n)-H\le mn+z(m,n).
\]
Therefore
\begin{equation}\label{eq:universal}
z_2(m,n)\le \left\lfloor\frac{mn+z(m,n)}2\right\rfloor.
\end{equation}
This elementary bound reduces the later $5\times5$ problem to excluding only the 18-square case.

\section{The $4\times4$ case: a complete upper-bound proof}\label{sec:44}

We begin with the smallest nontrivial square case. The goal is not merely to obtain a numerical value, but to exhibit the basic pattern used repeatedly later: fix the ordinary $C_4$-free extremal skeleton, then use strip bounds and product relations to eliminate possible two-edge positions. Previous work gives $\zRL(4,4)=10$ \cite{LQ}, hence $z_2(4,4)\ge10$, while the universal cell bound gives only $z_2(4,4)\le12$. We prove the sharp upper bound $z_2(4,4)\le10$ and determine the uniqueness of the extremal configuration.

\subsection{The ordinary extremal skeleton}

\begin{lemma}\label{lem:44base}
We have $z(4,4)=9$. Up to row and column permutations, there is a unique $C_4$-free $4\times4$ bipartite graph with nine one-edges.
\end{lemma}
\begin{proof}
Let the four row degrees be $d_0,d_1,d_2,d_3$. Since the graph is $C_4$-free, any pair of columns can occur together in at most one row, so
\[
\sum_{i=0}^{3}\binom{d_i}{2}\le\binom42=6.
\]
If there were at least ten edges, then among integer degree vectors with total at least ten, the left-hand side would be at least the value for the degree pattern $(3,3,2,2)$, namely
\[
3+3+1+1=8,
\]
a contradiction. Hence there are at most nine edges. The following configuration attains nine:
\[
E_1=\{(0,1),(0,2),(0,3),(1,0),(1,1),(2,0),(2,2),(3,0),(3,3)\}.
\]
Thus $z(4,4)=9$. At equality the row degrees sum to nine and $\sum\binom{d_i}{2}=6$, so the only possible degree pattern is $(3,2,2,2)$. After permuting rows and columns, let row~0 have neighborhood $\{1,2,3\}$. This row already uses the column pairs $12,13,23$, so the three remaining degree-two rows cannot use any of these pairs and must instead use $01,02,03$, respectively. Hence the isomorphism type is unique, with grid
\[
\begin{array}{c|cccc}
 &0&1&2&3\\\hline
0&\hole&\one&\one&\one\\
1&\one&\one&\hole&\hole\\
2&\one&\hole&\one&\hole\\
3&\one&\hole&\hole&\one
\end{array}
\]
\end{proof}

\subsection{Two-edge exclusion and the exact value}

Let $u_{ij}=x_i y_j$. After fixing the ordinary skeleton above, the seven free cells are
\[
u_{00},\quad u_{12},u_{13},\quad u_{21},u_{23},\quad u_{31},u_{32}.
\]

\begin{theorem}\label{thm:44}
\[
z_2(4,4)=\zSL(4,4)=\zRL(4,4)=10.
\]
Moreover, the irreducible limited $4\times4$ configurations attaining ten displayed squares form a single isomorphism class under row and column permutations.
\end{theorem}

\begin{proof}
For the lower bound, take the nine one-edges of Lemma~\ref{lem:44base} and add the two-edge $u_{12}+u_{23}$. The displayed decomposition has ten squares. The closure generated by Definition~\ref{def:rw3rules} identifies the two halves of this two-edge and proves mutual orthogonality of the ten displayed forms, so the configuration is $(RW3^+)$-admissible. The complete closure derivation is included in Supplementary Data~$\mathrm{S1}$ and is replayed step by step by an independent checker. Hence $z_2(4,4)\ge\zRL(4,4)\ge10$; the value $\zRL(4,4)=10$ is also known from \cite{LQ}. We now prove the upper bound $z_2(4,4)\le10$.

Let $G$ be an irreducible limited $4\times4$ configuration. By Lemma~\ref{lem:delete}, if a two-edge can occur in $G$, then deleting all other two-edges leaves an irreducible sub-decomposition consisting of the nine one-edges together with that one two-edge. Thus the $\binom72=21$ possible pairs of free cells can be analyzed one at a time.

First suppose the two-edge contains $u_{00}$, with the other cell $u_{ij}$, where $i,j\in\{1,2,3\}$ and $i\ne j$. On rows $\{0,i\}$ and columns $\{0,i,j\}$ there are already four one-edge squares, and the new two-edge adds a fifth square. Lemma~\ref{lem:strip} gives an upper bound of four for this $2\times3$ strip, so the sub-sum is reducible. All six pairs of this type are excluded.

Second, if the two free cells lie in a common row or column, for example $u_{ij}+u_{ik}$, then rows $\{0,i\}$ and the corresponding three columns again form a five-square $2\times3$ strip; the column case is the transpose. All six such pairs are excluded.

Third, if the two free cells are reversed positions $u_{ij}+u_{ji}$, then rows $\{i,j\}$ and columns $\{0,i,j\}$ again contain a five-square $2\times3$ strip. These three pairs are excluded.

Thus only six of the 21 possibilities remain:
\begin{equation}\label{eq:44six}
u_{ij}+u_{jk},\qquad \{i,j,k\}=\{1,2,3\}.
\end{equation}

We next prove that two two-edges cannot occur simultaneously. Using the $S_3$ symmetry on $1,2,3$, normalize the first two-edge to
\[
A=u_{12}+u_{23}.
\]
The only second candidates disjoint from $A$ are
\[
u_{13}+u_{32},\qquad u_{21}+u_{13},\qquad B=u_{32}+u_{21}.
\]
The first two alternatives give six-square sub-sums on a $4\times2$ or $2\times4$ strip, respectively, whereas Lemma~\ref{lem:strip} gives an upper bound of five. Hence both are reducible.

Only the pair $A,B$ remains. Set
\[
a=u_{01},\quad b=u_{02},\quad c=u_{03},\quad d=u_{11},\quad h=u_{33}.
\]
Using the rectangle product identities
\[
u_{01}u_{12}=u_{02}u_{11},\qquad
u_{01}u_{23}=u_{03}u_{21},\qquad
u_{03}u_{32}=u_{02}u_{33},
\]
we obtain
\begin{equation}\label{eq:44rel}
aA-cB-bd+bh=0.
\end{equation}
By Lemma~\ref{lem:product}, the corresponding displayed decomposition is reducible. More explicitly, the same product relation yields the compression
\begin{equation}\label{eq:44compress}
\begin{aligned}
&a^2+b^2+c^2+d^2+h^2+A^2+B^2\\
&\qquad=\Bigl(a-\frac{A}{\sqrt2}\Bigr)^2
+\Bigl(\frac{A}{\sqrt2}\Bigr)^2
+\Bigl(c+\frac{B}{\sqrt2}\Bigr)^2
+\Bigl(\frac{B}{\sqrt2}\Bigr)^2\\
&\qquad\quad+\Bigl(d+\frac{b}{\sqrt2}\Bigr)^2
+\Bigl(h-\frac{b}{\sqrt2}\Bigr)^2.
\end{aligned}
\end{equation}
Seven squares have been replaced by six. Therefore an irreducible configuration contains at most one two-edge, and hence
\[
z_2(4,4)\le9+1=10.
\]
Combining with the lower bound gives equality, and the hierarchy yields equality of all three parameters.

Finally, equality requires the nine one-edges together with one of the six two-edges in \eqref{eq:44six}. These six choices are equivalent under simultaneous permutations of $1,2,3$, so the extremal isomorphism class is unique. We may take $u_{12}+u_{23}$ as representative; its $(RW3^+)$ closure is verified according to Definition~\ref{def:rw3}, with the complete derivation stored in Supplementary Data~$\mathrm{S1}$.
\end{proof}

\section{Four-column extension structure: from $6\times4$ to $8\times4$}\label{sec:fourchain}

The previous section fixed the ordinary skeleton and then excluded two-edge placements exactly. For four columns and $m\ge6$, the ordinary skeleton has an even more stable structure, allowing a chain of the form ``classify lower-order extremizers -- delete one row -- exclude the next order. By the classical results of \v{C}ul\'ik and Reiman \cite{Culik,Reiman},
\begin{equation}\label{eq:zm4}
z(m,4)=m+6\qquad(m\ge6),
\end{equation}
and every ordinary extremal skeleton consists of six degree-two rows and $m-6$ degree-one rows; the six degree-two rows correspond to the column pairs
\[
12,13,14,23,24,34.
\]
We denote them throughout by
\begin{equation}\label{eq:corelabels}
A=12,\quad B=13,\quad C=14,\quad D=23,\quad E=24,\quad F=34.
\end{equation}

By \eqref{eq:cellcount} and $|E_1|=m+6$, every limited four-column configuration satisfies
\begin{equation}\label{eq:fourcell}
R(G)=\frac{5m+6-H}{2},\qquad
z_2(m,4)\le\left\lfloor\frac{5m+6}{2}\right\rfloor.
\end{equation}

\subsection{The $6\times4$ case: the exact value and the unique extremal structure}

Previous work proved $z_2(6,4)=\zSL(6,4)=\zRL(6,4)=16$ \cite{LQ}. The new point here is the classification of extremal configurations. The twelve free cells in the six core rows are
\[
A3,A4;\ B2,B4;\ C2,C3;\ D1,D4;\ E1,E3;\ F1,F2.
\]
Sixteen displayed squares mean twelve one-edges and four two-edges.

\begin{theorem}\label{thm:64class}
Up to row and column permutations, there is a unique irreducible limited $6\times4$ configuration attaining $z_2(6,4)=16$. A representative is
\begin{equation}\label{eq:64pairs}
A3+B4,\qquad B2+F1,\qquad C2+E3,\qquad D4+E1.
\end{equation}
The corresponding grid is
\begin{equation}\label{eq:64grid}
\begin{array}{c|cccc}
 &1&2&3&4\\\hline
A&\one&\one&a&\hole\\
B&\one&b&\one&a\\
C&\one&c&\hole&\one\\
D&\hole&\one&\one&d\\
E&d&\one&c&\one\\
F&b&\hole&\one&\one
\end{array}
\end{equation}
\end{theorem}

\begin{proof}
There are $\binom{12}{2}=66$ single two-edge candidates on the twelve free cells. Two distinct core rows either share one column or are disjoint. Same-column pairings are excluded by two-column strip bounds; reversed choices on intersecting rows are excluded by a $2\times3$ strip; and for two disjoint rows one can adjoin the core row corresponding to the relevant column pair to obtain a $3\times2$ overload. Thus strip bounds eliminate all unsuitable single candidates, leaving 30: six row-degenerate two-edges and 24 nondegenerate candidates, all lying in the column-permutation orbit of $A3+B4$.

Among these 30 retained candidates, overlapping supports are incompatible. All remaining necessary incompatibilities are generated by the following table and all column permutations.
\begin{center}\small
\begin{tabular}{@{}lll r@{}}
\toprule
First pair & Second pair & Exclusion & Orbit size\\
\midrule
$A3+A4$ & $B2+B4$ & two-row, four-column overload & 12\\
$A3+A4$ & $F1+F2$ & two-row, four-column overload & 3\\
$A3+A4$ & $B2+C3$ & identity \eqref{eq:64R} & 24\\
$A3+B4$ & $A4+B2$ & two-row, four-column overload & 12\\
$A3+B4$ & $A4+C3$ & overload on rows $A,B,C,F$ and columns $3,4$ & 12\\
$A3+B4$ & $A4+E3$ & overload on rows $A,B,E,F$ and columns $3,4$ & 12\\
$A3+B4$ & $A4+D1$ & identity \eqref{eq:S} & 24\\
$A3+B4$ & $B2+C3$ & identity \eqref{eq:64N} & 12\\
\bottomrule
\end{tabular}
\end{center}
Let $R=A3+A4$, $N=A3+B4$, and $M=B2+C3$. Then
\begin{align}
(A1)(C4)-(C1)R+(A1)M-(A2)(B1)&=0,\label{eq:64R}\\
(A1)M-(C1)N-(A2)(B1)+(B1)(C4)&=0.\label{eq:64N}
\end{align}
Identity~\eqref{eq:S} is used with $(x,y,z)=(D,B,A)$ and $(a,b,d)=(4,1,3)$. Every strip case not written explicitly in the table contains six complete displayed squares, while Lemma~\ref{lem:strip} gives an upper bound of five.

Suppose now that the four selected two-edges include $R=A3+A4$. It has only the following twelve possible neighbors:
\[
\begin{array}{rrrr}
B2+D4,&B2+F1,&B4+D1,&B4+F2,\\
C2+E3,&C2+F1,&C3+E1,&C3+F2,\\
D1+F2,&D4+F1,&E1+F2,&E3+F1.
\end{array}
\]
With the numbering shown, the necessary compatibility graph is bipartite between $\{1,2,3,4,9,10\}$ and $\{5,6,7,8,11,12\}$. Hence it contains no triangle and cannot supply three neighbors of $R$. Therefore an extremal four-two-edge family contains no row-degenerate two-edge.

Normalize a nondegenerate candidate to $N=A3+B4$. The remaining nondegenerate neighbors and their possible larger-index neighbors are as follows.
\begin{center}\small
\begin{tabular}{@{}rll@{}}
\toprule
Index & Candidate & Compatible larger-index neighbors\\
\midrule
1 & $B2+F1$ & 2,4,6,8\\
2 & $C2+E3$ & 7,8,9\\
3 & $C2+F1$ & 6,8\\
4 & $C3+E1$ & 7,9\\
5 & $C3+F2$ & 6,8,9\\
6 & $D1+E3$ & none\\
7 & $D1+F2$ & none\\
8 & $D4+E1$ & none\\
9 & $D4+F1$ & none\\
10 & $E1+F2$ & none\\
11 & $E3+F1$ & none\\
\bottomrule
\end{tabular}
\end{center}
The unique triangle is $\{1,2,8\}$, which gives precisely \eqref{eq:64pairs}; there is no four-clique. Hence at most four two-edges are possible, and equality can occur only for the displayed representative. For this representative, the closure of Definition~\ref{def:rw3rules} identifies all four two-edges and proves mutual orthogonality of the sixteen displayed forms, so the configuration is $(RW3^+)$-admissible and therefore irreducible. The full closure derivation is included in Supplementary Data~$\mathrm{S1}$ and replayed independently. For the fixed ordinary skeleton there are six labeled representatives, forming a single column-permutation orbit, so the extremal isomorphism class is unique. The core remains equivalent under the column cycle $1\mapsto2\mapsto3\mapsto4\mapsto1$, so its automorphism group is transitive on the four columns; this will be used to normalize the new one-edge in the next subsection.
\end{proof}

\subsection{The $7\times4$ case: excluding 20 squares from the six-row classification}

The ordinary extremal $7\times4$ skeleton consists of the six-row core plus one degree-one row, denoted by $h$. The value $\zRL(7,4)=19$ is known \cite{CC}, so $z_2(7,4)\ge19$, while the cell bound gives only $z_2(7,4)\le20$. We now exclude 20 displayed squares. This argument uses only the equality classification in six rows and does not presuppose the later complete classification of seven-row 19-square configurations.

\begin{theorem}\label{thm:74value}
\[
z_2(7,4)=\zSL(7,4)=\zRL(7,4)=19.
\]
\end{theorem}

\begin{proof}
Assume for contradiction that an irreducible limited 20-square configuration exists. It has 13 one-edges and seven two-edges. The degree-one row $h$ has one one-edge and three free cells. Deleting every complete displayed square whose support meets $h$ removes at most the one-edge square and three two-edge squares, hence at most four squares. By Lemma~\ref{lem:delete}, the six-row restriction remains irreducible and has at least $20-4=16$ displayed squares. But the six-row bound \eqref{eq:fourcell} gives $z_2(6,4)\le16$, so equality must hold. Theorem~\ref{thm:64class} forces the restriction to be isomorphic to the unique extremal type \eqref{eq:64grid}.

Therefore each of the three free cells in row $h$ must belong to a distinct two-edge crossing into the old core. If one were a hole, or if two were paired within row $h$, then fewer than three two-edges would meet $h$, and deletion would leave more than 16 squares in six rows, contradicting $z_2(6,4)=16$. Since the core is equivalent under the column cycle $1\mapsto2\mapsto3\mapsto4\mapsto1$, we may place the new one-edge at $h1$. The cells $h2,h3,h4$ must then be injected into the four old holes
\[
A4,\qquad C3,\qquad D1,\qquad F2.
\]
The following possible partners of $h2$ and $h3$ are excluded:
\begin{center}\small
\begin{tabular}{@{}lll@{}}
\toprule
New cell & Excluded partner & Reason\\
\midrule
$h2$ & $C3$ & identity \eqref{eq:74rel}\\
$h2$ & $D1$ & five-square overload on rows $A,D,h$ and columns $1,2$\\
$h2$ & $F2$ & six-square overload on rows $A,B,F,h$ and columns $1,2$\\
$h3$ & $C3$ & six-square overload on rows $A,B,C,h$ and columns $1,3$\\
$h3$ & $D1$ & six-square overload on rows $A,B,D,h$ and columns $1,3$\\
$h3$ & $F2$ & identity \eqref{eq:74short}\\
\bottomrule
\end{tabular}
\end{center}
where
\begin{equation}\label{eq:74rel}
(A1)(h2+C3)-(A2)(h1)-(C1)(A3+B4)+(B1)(C4)=0,
\end{equation}
and the final exclusion follows from the identity
\begin{align}
&(h3+F2)^2+(B2+F1)^2+(h1)^2+(B3)^2+(B1)^2\notag\\
&\qquad=(h3+B1+F2)^2+(h1-B3)^2+(B2)^2+(F1)^2.\label{eq:74short}
\end{align}
Thus both $h2$ and $h3$ would have to use the same old hole $A4$, violating simplicity. Hence no irreducible limited 20-square configuration exists, so $z_2(7,4)\le19$. Together with the known value $\zRL(7,4)=19$ \cite{CC} and the hierarchy \eqref{eq:hierarchy}, this proves equality.
\end{proof}

\subsection{Extremal 19-square configurations on $7\times4$}

Theorem~\ref{thm:74value} gives the exact seven-row value but is not sufficient for the eight-row case. If deleting a degree-one row from an eight-row candidate leaves exactly 19 squares, there is no numerical contradiction. To continue the exclusion, one must know which 19-square configurations can actually occur. We therefore classify all irreducible 19-square $7\times4$ configurations up to isomorphism.

After fixing the new one-edge to be $h1$, there are 15 free cells. Nineteen displayed squares require six support-disjoint two-edges, leaving three holes. Thus the number of labeled pair families before pruning is
\begin{equation}\label{eq:74raw}
M(15,6)=\frac{15!}{3!\,2^6\,6!}=4\,729\,725.
\end{equation}
Directly testing all families would be both opaque and unnecessary. We follow Algorithm~\ref{alg:finite}: prune single two-edge candidates locally, construct the necessary compatibility graph and enumerate all 6-cliques, quotient by the automorphism group of the ordinary skeleton, and finally attach a verifiable reducibility or irreducibility proof to every orbit.

\begin{theorem}\label{thm:74classes}
Up to row and column relabeling, the irreducible limited $7\times4$ configurations attaining $z_2(7,4)=19$ form exactly three isomorphism classes. Representatives are shown below; $\one$ denotes a one-edge, $\hole$ a hole, and two occurrences of the same letter one two-edge:
\begin{center}\small
\begin{tabular}{ccc}
{\normalfont\bfseries Type I}&{\normalfont\bfseries Type II}&{\normalfont\bfseries Type III}\\[2mm]
$\begin{array}{c|cccc}
 &1&2&3&4\\\hline
A&\one&\one&a&a\\
B&\one&b&\one&\hole\\
C&\one&c&d&\one\\
D&e&\one&\one&b\\
E&d&\one&\hole&\one\\
F&f&\hole&\one&\one\\
h&\one&f&c&e
\end{array}$
&
$\begin{array}{c|cccc}
 &1&2&3&4\\\hline
A&\one&\one&a&a\\
B&\one&b&\one&\hole\\
C&\one&c&d&\one\\
D&e&\one&\one&b\\
E&d&\one&\hole&\one\\
F&\hole&f&\one&\one\\
h&\one&f&c&e
\end{array}$
&
$\begin{array}{c|cccc}
 &1&2&3&4\\\hline
A&\one&\one&a&a\\
B&\one&b&\one&c\\
C&\one&\hole&d&\one\\
D&c&\one&\one&\hole\\
E&e&\one&\hole&\one\\
F&f&d&\one&\one\\
h&\one&f&e&b
\end{array}$
\end{tabular}
\end{center}
\end{theorem}

\begin{proof}
Fix the ordinary extremal skeleton and place the unique one-edge in the degree-one row at $h1$. The 15 free cells give $\binom{15}{2}=105$ single two-edge candidates. We apply Algorithm~\ref{alg:finite} with target length 19, so $k=6$.

Local reductions eliminate 48 single candidates, all by two-row or two-column strip overload, leaving 57. Build the necessary compatibility graph on these 57 candidates: two vertices are adjacent only if their supports are disjoint and their joint occurrence has not already been proved reducible locally. The two-candidate exclusions comprise 390 overlapping-support pairs, 84 strip overloads, 54 product relations, 48 instances of \eqref{eq:S}, and 6 instances of \eqref{eq:E}, leaving 1014 graph edges. The graph contains 1114 six-cliques. Every irreducible six-two-edge family must be one of these cliques. Applying larger strip bounds to the six-cliques excludes 120 more families, leaving 994 labeled families.

The automorphism group of the fixed ordinary skeleton consists of the six permutations of columns $2,3,4$ together with the induced permutations of the core rows; hence $\Gamma\cong S_3$ and $|\Gamma|=6$. The 994 labeled families form 170 orbits under $\Gamma$. Of these, 86 orbits, corresponding to 506 labeled families, are excluded directly by product relations in the original displayed space. This leaves 84 orbits, each of which is resolved by one of the verifiable proofs described in Section~\ref{subsec:certs}:
\begin{center}\small
\begin{tabular}{@{}lrr@{}}
\toprule
Proof type & Orbits & Labeled families\\
\midrule
Explicit shorter integer SOS & 44 & 256\\
Product relation after an equal-length rewrite & 10 & 56\\
Low-rank Gram matrix over $\mathbb Q(\sqrt2)$ & 2 & 12\\
Low-rank Gram matrix from rational contraction & 25 & 146\\
Irreducible representatives verified by full $(RW3^+)$ closure & 3 & 18\\
\bottomrule
\end{tabular}
\end{center}
All Gram representations for negative cases have rank strictly below 19. Examples~\ref{ex:intsos} and~\ref{ex:switch} give short representatives of the first two proof types; the exact quadratic-field Gram matrices are verified by Lemma~\ref{lem:qgram}, and the rational contraction records by Lemma~\ref{lem:contraction}. Hence every 19-square candidate outside the three positive orbits is reducible.

All three positive representatives are $(RW3^+)$-admissible by Definition~\ref{def:rw3}: the closure identifies all six two-edges and proves mutual orthogonality between distinct displayed forms, so they are irreducible. The complete closure derivations are included in Supplementary Data~$\mathrm{S1}$ and replayed by the independent checker. For reference, the two-edges of the three representatives are
\begin{align*}
\mathrm{I}:\quad& A3+A4,\ B2+D4,\ C2+h3,\ C3+E1,\ D1+h4,\ F1+h2;\\
\mathrm{II}:\quad& A3+A4,\ B2+D4,\ C2+h3,\ C3+E1,\ D1+h4,\ F2+h2;\\
\mathrm{III}:\quad& A3+A4,\ B2+h4,\ B4+D1,\ C3+F2,\ E1+h3,\ F1+h2.
\end{align*}
The finite classification leaves 18 labeled irreducible configurations. Each of the three representatives has orbit size six, and the three orbits are disjoint, so the 18 configurations form exactly three isomorphism classes. Types I and II are distinguished by the presence of a column-degenerate two-edge. Types I and III are also inequivalent: any isomorphism must preserve the unique degree-one row and its ordinary column~1, and must send the row $A=12$ containing the unique row-degenerate pair back to itself; hence column~2 is fixed. Only the identity and the transposition of columns 3 and 4 remain, and neither maps Type I to Type III.

It is worth emphasizing that deleting row $h$ from any of the three extremal classes leaves only 15 squares, not the six-row extremal value 16. Thus uniqueness of the six-row extremal template must not be misread as saying that the internal six rows of every larger extremizer must themselves equal that template.
\end{proof}

\subsection{The $8\times4$ case: excluding 22 squares using the seven-row classification}

The ordinary extremal $8\times4$ skeleton consists of the six core rows and two degree-one rows. Up to column permutation, the two one-edges in the degree-one rows have only two essential position types: the same column or different columns. Denote the two ordinary skeletons by
\[
\mathcal A_8:\ h1,i1,
\qquad
\mathcal B_8:\ h1,i2.
\]
The value $\zRL(8,4)=21$ is known \cite{CC}, hence $z_2(8,4)\ge21$. For an explicit lower-bound witness, start from Type I of Theorem~\ref{thm:74classes}, add a row $i$, take the one-edge $i1$, and add the two-edge $B4+i3$. This gives
\begin{equation}\label{eq:lower8}
\begin{array}{c|cccc}
 &1&2&3&4\\\hline
A&\one&\one&a&a\\
B&\one&b&\one&g\\
C&\one&c&d&\one\\
D&e&\one&\one&b\\
E&d&\one&\hole&\one\\
F&f&\hole&\one&\one\\
h&\one&f&c&e\\
i&\one&\hole&g&\hole
\end{array}
\end{equation}
It has 14 one-edges, seven two-edges, and four holes. It is $(RW3^+)$-admissible by Definition~\ref{def:rw3}, and hence is an irreducible limited 21-square configuration. The complete closure derivation is contained in Supplementary Data~$\mathrm{S1}$ and replayed step by step by the independent checker.

By \eqref{eq:fourcell}, eight rows allow at most 23 displayed squares. If an irreducible 23-square configuration existed, deleting one complete two-edge square would produce an irreducible 22-square configuration. Thus it suffices to exclude 22 squares. Each ordinary skeleton has 18 free cells. A 22-square configuration consists of 14 one-edges and eight two-edges, leaving two holes; hence each skeleton has the following number of labeled pair families before pruning:
\begin{equation}\label{eq:84raw}
M(18,8)=\frac{18!}{2!\,2^8\,8!}=310\,134\,825.
\end{equation}

\begin{theorem}\label{thm:84}
\[
z_2(8,4)=\zSL(8,4)=\zRL(8,4)=21.
\]
\end{theorem}

\begin{proof}
Apply Algorithm~\ref{alg:finite} separately to $\mathcal A_8$ and $\mathcal B_8$ with target length 22, so $k=8$. The stabilizer automorphism groups of the two ordinary skeletons have orders 12 and 4, respectively.

We first generate the single two-edge candidates and the necessary compatibility graph and enumerate all 8-cliques. Obvious reducible cliques are removed by strip bounds, the remainder is quotiented by the automorphism group of the ordinary skeleton, and direct product relations eliminate another layer. The exact counts are as follows.
\begin{center}\small
\begin{tabular}{@{}lrr@{}}
\toprule
Stage & $\mathcal A_8$ & $\mathcal B_8$\\
\midrule
All single two-edge candidates & 153 & 153\\
Retained after local exclusions & 90 & 92\\
Edges in the necessary compatibility graph & 2877 & 2998\\
Number of 8-cliques & 28086 & 41598\\
After further strip-overload exclusions & 24366 & 33386\\
Orbits remaining after original-space product relations & 829 & 3758\\
\bottomrule
\end{tabular}
\end{center}
For $\mathcal A_8$, the two-candidate exclusions consist of 123 strip overloads, 81 product relations, 72 instances of \eqref{eq:S}, and 12 instances of \eqref{eq:E}; for $\mathcal B_8$, the corresponding counts are 143, 73, 80, and 12.

For each remaining orbit, delete each of the two degree-one rows in turn. If the resulting seven-row restriction has more than 19 squares, it contradicts Theorem~\ref{thm:74value}; if it has exactly 19, then it must be isomorphic to one of Types I--III in Theorem~\ref{thm:74classes}, otherwise it is excluded. Any orbit not eliminated this way is ruled out by an explicit shorter SOS, a product relation after an equal-length rewrite, or a positive semidefinite Gram representation of rank at most 21 verified by Lemma~\ref{lem:qgram} or Lemma~\ref{lem:contraction}. The remaining orbit counts are
\begin{center}\small
\begin{tabular}{@{}lrr@{}}
\toprule
Reduction proof for the orbit & $\mathcal A_8$ & $\mathcal B_8$\\
\midrule
Seven-row upper bound or three-class extremal classification & 283 & 1365\\
Explicit shorter integer SOS & 434 & 1880\\
Product relation after an equal-length rewrite & 71 & 265\\
Low-rank positive semidefinite Gram representation & 41 & 248\\
\midrule
Total & 829 & 3758\\
\bottomrule
\end{tabular}
\end{center}
The first class of exclusions uses the already proved value $z_2(7,4)=19$ and the complete three-class classification, not the failure of any recursive test. Thus all 22-square orbits on both ordinary skeletons are excluded. If an irreducible 23-square configuration existed, deleting one complete two-edge square would yield an irreducible 22-square configuration, a contradiction. Since \eqref{eq:fourcell} already rules out more than 23 squares, we obtain $z_2(8,4)\le21$. Combining this with \eqref{eq:lower8}, the known value $\zRL(8,4)=21$, and the hierarchy \eqref{eq:hierarchy} gives equality.
\end{proof}

\subsection{Structural extension to general $m\times4$}

The $6\to7\to8$ chain can be abstracted into a deletion inequality.

\begin{proposition}\label{prop:fourdelete}
Let $m\ge7$, and let $G$ be an irreducible limited $m\times4$ configuration with $R$ displayed squares. For a degree-one row $u$, let $h_u$ be the number of holes in that row and let $p_u$ be the number of two-edges whose two halves both lie in that row. Then
\begin{equation}\label{eq:fourdelete}
R-4+h_u+p_u\le z_2(m-1,4).
\end{equation}
In particular,
\[
z_2(m,4)\le z_2(m-1,4)+4.
\]
\end{proposition}
\begin{proof}
Row $u$ contains one one-edge and three free cells. Let $x_u$ be the number of two-edges having exactly one half in that row. Then
\[
3=h_u+2p_u+x_u.
\]
Deleting all complete displayed squares meeting the row removes
\[
1+p_u+x_u=4-h_u-p_u
\]
squares. The remaining ordinary skeleton is still an extremal $(m-1)\times4$ skeleton, and the remainder is irreducible by Lemma~\ref{lem:delete}. Hence its length is at most $z_2(m-1,4)$, giving \eqref{eq:fourdelete}; the stated corollary follows from $h_u,p_u\ge0$.
\end{proof}

Inequality~\eqref{eq:fourdelete} is usually weaker numerically than the cell bound. Its real value is structural near equality. If deletion leaves exactly $z_2(m-1,4)$ squares, the remainder must belong to an extremal isomorphism class on $(m-1)\times4$. Thus a natural next step toward $9\times4$ is not a blind enumeration of all nine-row pairings, but first a classification of all irreducible 21-square extremizers on $8\times4$, followed by restriction of near-extremal nine-row candidates to those templates.

We now extend this structure to a fixed number of columns. Let $\beta=\binom n2$.

\begin{lemma}[Ordinary extremal skeleton at fixed width]\label{lem:fixedwidth}
If $m\ge\beta$, then
\[
z(m,n)=m+\beta.
\]
Moreover, every $C_4$-free ordinary skeleton attaining this bound has exactly $\beta$ degree-two rows, one for each pair of the $n$ columns, and the remaining $m-\beta$ rows all have degree one.
\end{lemma}
\begin{proof}
Let the row degrees be $d_1,\ldots,d_m$. Since the graph is $C_4$-free, each pair of columns can occur together in at most one row, so
\[
\sum_{i=1}^{m}\binom{d_i}{2}\le\binom n2=\beta.
\]
Also, for every integer $d_i\ge0$, we have $d_i-1\le\binom{d_i}{2}$, so
\[
|E_1|-m=\sum_{i=1}^{m}(d_i-1)
\le\sum_{i=1}^{m}\binom{d_i}{2}
\le\beta,
\]
Hence $|E_1|\le m+\beta$. When $m\ge\beta$, choose $\beta$ degree-two rows representing all column pairs and add $m-\beta$ arbitrary degree-one rows; this attains $m+\beta$. If equality holds, both inequalities above are equalities. Thus every row has degree one or two, and every pair of columns occurs exactly once. Therefore there are exactly $\beta$ degree-two rows and all remaining rows have degree one.
\end{proof}

Choose a set $U$ of $s$ degree-one rows. Let $h(U)$ be the number of holes in these rows, $p(U)$ the number of two-edges with both halves in $U$, and $x(U)$ the number of two-edges with exactly one half in $U$. The $s$ rows contain $ns$ cells, of which $s$ are occupied by one-edges, so
\[
ns-s=h(U)+2p(U)+x(U).
\]
Deleting every complete square meeting $U$ removes $s+p(U)+x(U)=ns-h(U)-p(U)$ displayed squares. The remainder is still an irreducible limited $(m-s)\times n$ configuration, hence
\begin{equation}\label{eq:generaldelete}
R-ns+h(U)+p(U)\le z_2(m-s,n).
\end{equation}
This shows that exact small-size values and extremal classifications propagate structural information to larger sizes. It is not, however, an automatic recurrence for all $m$: when deletion attains equality, the structure of the lower-order extremal configurations is still needed.

\section{The $5\times5$ case: finite classification and a ten-to-nine-square compression}\label{sec:55}

The four-column method relies on the rigidity of the ordinary extremal skeleton at fixed width. To test whether the upper-bound ideas extend beyond four columns, we study the balanced square case $5\times5$. Here the ordinary $C_4$-free extremal skeleton is no longer unique, but the universal cell bound is only one above the known value $\zRL(5,5)$, so only 18-square candidates need to be excluded. This case also exhibits a phenomenon different from the four-column chain: local strip arguments and product relations eliminate almost all candidates, but the final step requires a nine-square representation using bilinear forms outside the span of the original displayed forms.

\subsection{The two isomorphism types of ordinary extremal skeletons}

\begin{lemma}\label{lem:55base}
We have $z(5,5)=12$. Up to row and column permutations, there are exactly two isomorphism types of $C_4$-free $5\times5$ bipartite graphs with 12 one-edges; representatives are
\begin{align}
\mathcal A_5:\quad E_1&=\{(0,j),(j,0),(j,j):1\le j\le4\},\label{eq:A55}\\
\mathcal B_5:\quad E_1&=\{(0,0),(0,1),(0,2),(1,0),(1,3),(1,4),\notag\\
&\hspace{20mm}(2,1),(2,3),(3,1),(3,4),(4,2),(4,3)\}.\label{eq:B55}
\end{align}
\end{lemma}
\begin{proof}
Let the row degrees be $d_0,\ldots,d_4$. Since the graph is $C_4$-free, the neighborhoods of any two rows intersect in at most one column, giving the column-pair count
\[
\sum_{i=0}^{4}\binom{d_i}{2}\le\binom52=10.
\]
With thirteen edges the left-hand side is at least 11, so there are at most twelve edges. When the total degree is twelve, the only nonincreasing degree patterns satisfying the bound are
\[
(3,3,3,2,1),\qquad (4,2,2,2,2),\qquad (3,3,2,2,2).
\]
The first pattern is impossible: three 3-element neighborhoods with pairwise intersections of size at most one have union size at least $9-3=6$, which cannot fit into five columns.

For the degree pattern $(4,2,2,2,2)$, label the degree-four row as row~0; its neighborhood is $\{1,2,3,4\}$. Every degree-two row must contain column~0, otherwise it would choose two columns from $\{1,2,3,4\}$ and form a $C_4$ with row~0. The second neighbors of the four degree-two rows must then run through $\{1,2,3,4\}$ exactly once; otherwise a repeated choice again creates a $C_4$. Label these rows $1,2,3,4$ according to that second neighbor. This gives $\mathcal A_5$, uniquely up to row and column permutations.

For the degree pattern $(3,3,2,2,2)$, the two 3-element neighborhoods must intersect in exactly one column: if disjoint, their union has six columns; if they intersect in two columns, they form a $C_4$. Label these rows $0,1$, call the common column 0, and set
\[
N(0)=\{0,1,2\},\qquad N(1)=\{0,3,4\}.
\]
Each of the remaining three rows has degree two, with one neighbor in $\{1,2\}$ and one in $\{3,4\}$; otherwise it forms a $C_4$ with row~0 or row~1. Thus the three rows correspond to three distinct cross edges of the complete bipartite graph $K_{2,2}$ between $\{1,2\}$ and $\{3,4\}$, so exactly one cross edge is omitted. The four choices of the omitted edge are equivalent under swapping columns $1,2$, swapping columns $3,4$, and simultaneously interchanging the two column blocks together with rows $0,1$. Hence this yields a single isomorphism class $\mathcal B_5$. Both displayed skeletons are $C_4$-free and have twelve edges.
\end{proof}

The value $\zRL(5,5)=17$ is known \cite{LQ}, so $z_2(5,5)\ge17$, while \eqref{eq:universal} gives
\[
z_2(5,5)\le\left\lfloor\frac{25+12}{2}\right\rfloor=18.
\]
Thus only irreducible 18-square configurations need to be excluded. Such a configuration has twelve one-edges, six two-edges, and one hole among the thirteen free cells.

\subsection{Finite exclusion of the 18-square candidates}

Each ordinary skeleton has $\binom{13}{2}=78$ single two-edge candidates. Before pruning, the number of ways to select six support-disjoint two-edges is
\begin{equation}\label{eq:55raw}
M(13,6)=\frac{13!}{1!\,2^6\,6!}=270\,270
\end{equation}
labeled families.

\begin{lemma}\label{lem:55finite}
Apply Algorithm~\ref{alg:finite} to $\mathcal A_5$ and $\mathcal B_5$ with target length 18, so $k=6$. There is no irreducible 18-square configuration on $\mathcal B_5$; on $\mathcal A_5$, every six-pair family except the following two is excluded by strip overload, a product relation, or an instance of \eqref{eq:S}--\eqref{eq:E}:
\begin{align}
\mathcal F_+={}&\{u_{12}+u_{34},u_{13}+u_{42},u_{14}+u_{23},\notag\\
&\hspace{12mm}u_{21}+u_{43},u_{24}+u_{31},u_{32}+u_{41}\},\label{eq:Fplus}\\
\mathcal F_-={}&\{u_{12}+u_{43},u_{13}+u_{24},u_{14}+u_{32},\notag\\
&\hspace{12mm}u_{21}+u_{34},u_{23}+u_{41},u_{31}+u_{42}\}.\label{eq:Fminus}
\end{align}
\end{lemma}
\begin{proof}
The finite exclusion has the following size.
\begin{center}\small
\begin{tabular}{@{}lrr@{}}
\toprule
Quantity & $\mathcal A_5$ & $\mathcal B_5$\\
\midrule
All single two-edge candidates & 78 & 78\\
Retained after local exclusions & 36 & 37\\
Incompatible pairs with overlapping support & 180 & 180\\
Candidate pairs excluded by strip bounds & 24 & 50\\
Candidate pairs excluded by product relations & 72 & 34\\
Candidate pairs excluded by \eqref{eq:S} & 0 & 30\\
Candidate pairs excluded by \eqref{eq:E} & 0 & 8\\
Edges in the necessary compatibility graph & 354 & 364\\
Number of 6-cliques & 14 & 0\\
6-cliques further excluded by strip bounds & 12 & 0\\
Families not excluded by local methods & 2 & 0\\
\bottomrule
\end{tabular}
\end{center}
The necessary compatibility graph for $\mathcal B_5$ has no 6-clique, so that skeleton is completely excluded. Of the 14 six-cliques for $\mathcal A_5$, 12 are eliminated by larger strip overloads. The remaining two are precisely $\mathcal F_+$ and $\mathcal F_-$. We do not claim that an edge of the compatibility graph represents an irreducible pair without an $(RW3^+)$ assumption; only the implication ``excluded means genuinely reducible is used. The two remaining families are shown reducible in the next subsection because they define the same polynomial $T$, which has a nine-square representation; see Theorem~\ref{thm:T9}.
\end{proof}

\subsection{The two exceptions define the same polynomial}

Define
\begin{align}\label{eq:T}
T={}&\sum_{i=1}^{4}u_{ii}^2
+(u_{12}+u_{34})^2+(u_{13}+u_{42})^2+(u_{14}+u_{23})^2\notag\\
&+(u_{21}+u_{43})^2+(u_{24}+u_{31})^2+(u_{32}+u_{41})^2.
\end{align}
The six forms in $\mathcal F_-$ together with the four diagonal squares expand to the same polynomial $T$, because for four distinct indices
\[
u_{ij}u_{k\ell}=u_{i\ell}u_{kj}.
\]
Equivalently,
\begin{align}\label{eq:Texp}
T={}&\left(\sum_{i=1}^{4}x_i^2\right)\left(\sum_{j=1}^{4}y_j^2\right)\notag\\
&+2\bigl(x_1x_2y_3y_4+x_1x_3y_2y_4+x_1x_4y_2y_3\notag\\
&\hspace{11mm}+x_2x_3y_1y_4+x_2x_4y_1y_3+x_3x_4y_1y_2\bigr).
\end{align}

\subsection{An explicit nine-square representation}

\begin{theorem}\label{thm:T9}
The polynomial $T$ can be represented as a sum of nine squares of real bilinear forms. Hence the ten-square displayed decomposition in \eqref{eq:T} is reducible.
\end{theorem}

\begin{proof}
Let $q$ be the root of the cubic equation
\begin{equation}\label{eq:q}
2q^3+6q^2-1=0
\end{equation}
that lies uniquely in $(1/3,1/2)$, and set
\[
v=\frac{q}{1+q},\qquad \rho=\sqrt{(1-q)(1+v)},\qquad t=\frac{1+\rho}{2}.
\]
Then
\begin{equation}\label{eq:qrel}
(1+q)(1-v)=1,\qquad 1-2q=3v^2,
\qquad \rho^2=(1-q)(1+v).
\end{equation}
For bilinear forms $a,b,c,d$, define
\begin{align}
\mathcal B_+(a,b,c,d)={}&\frac12\left[\sqrt{1+q}(a+b)+\sqrt{1-v}(c+d)\right]^2\notag\\
&+\frac12\left[\sqrt{1-q}(a-b)+\sqrt{1+v}(c-d)\right]^2,\label{eq:Bplus}\\
\mathcal B_-(a,b,c,d)={}&\frac12\left[\sqrt{1-v}(a+b)+\sqrt{1+q}(c+d)\right]^2\notag\\
&+\frac12\left[\sqrt{1+v}(a-b)-\sqrt{1-q}(c-d)\right]^2.\label{eq:Bminus}
\end{align}
Taking $a=u_{11}$, $b=u_{22}$, $c=u_{33}$, and $d=u_{44}$, we have
\begin{align}\label{eq:T9}
T={}&\frac{1+q}{2}(a-b)^2+
\frac{1+q}{6}(a+b-2d)^2+
\bigl[c+v(a+b+d)\bigr]^2\notag\\
&+\mathcal B_+(u_{12},u_{21},u_{34},u_{43})\notag\\
&+\mathcal B_-(u_{13},u_{31},u_{24},u_{42})\notag\\
&+\mathcal B_+(u_{14},u_{41},u_{23},u_{32}).
\end{align}
The first line on the right contains three squares, and each of the following three terms contains two, for a total of nine squares. Expanding \eqref{eq:Bplus}--\eqref{eq:Bminus} and using \eqref{eq:qrel} shows that all diagonal terms agree with \eqref{eq:Texp}. Products of reversed cells satisfy $u_{ij}u_{ji}=u_{ii}u_{jj}$, and the rectangle products involving four distinct indices occur in pairs whose weights sum to $t+(1-t)=1$. Hence all cross terms also agree with \eqref{eq:Texp}, proving \eqref{eq:T9}.
\end{proof}

\begin{proposition}\label{prop:Trank}
Let $f_1,\ldots,f_{10}$ be the ten displayed forms in $T$, with the four diagonal forms $u_{11},u_{22},u_{33},u_{44}$ and the six two-edge forms from $\mathcal F_+$. Then the 45 mixed products $f_af_b$ for $1\le a<b\le10$ are linearly independent in the space of biquadratic monomials. In particular, there is no nonzero product relation $\sum_{a<b}c_{ab}f_af_b=0$ inside the span of the original displayed forms, so Lemma~\ref{lem:product} cannot produce a rank reduction there.
\end{proposition}
\begin{proof}
Expand each product $f_af_b$ as a sum of products of cell monomials. Encode the biquadratic monomial arising from a cell pair $\{(i,j),(k,\ell)\}$ by the ordered pair $(\{i,k\},\{j,\ell\})$. The 45 mixed products then form an integer matrix $M$ whose rows are indexed by these encodings and whose columns are indexed by unordered pairs $\{a,b\}$. Direct counting gives an $84\times45$ matrix. Gaussian elimination modulo each of $p=2,3,5,7,11,13,17,19$ gives $\rank_{\mathbb F_p}M=45$. In fact, any one of these primes is enough for a rigorous proof: modulo 2 the rank is 45, so some $45\times45$ minor is nonzero modulo 2; its integer determinant is therefore nonzero. Hence $M$ has column rank 45 over both $\mathbb Q$ and $\mathbb R$. The other primes serve only as independent cross-checks. Thus no nonzero real linear combination of the mixed products vanishes.

Because the supports of the ten displayed forms are pairwise disjoint, each $f_a^2$ contains a pure cell square that can arise only from that same diagonal Gram entry. Hence any coefficient-preserving Gram perturbation $H$ must satisfy $H_{aa}=0$. The remaining degrees of freedom are exactly the 45 mixed-product relations above. Their linear independence therefore shows that the Gram kernel within the span of the original displayed forms is zero.
\end{proof}

\begin{theorem}\label{thm:55}
\[
z_2(5,5)=\zSL(5,5)=\zRL(5,5)=17.
\]
\end{theorem}
\begin{proof}
The finite exclusion leaves only the two 18-square candidates $\mathcal F_+$ and $\mathcal F_-$. In both cases the full polynomial can be written as
\[
\sum_{j=1}^{4}(x_0y_j)^2+\sum_{i=1}^{4}(x_i y_0)^2+T.
\]
Keeping the first eight squares and replacing the ten-square displayed decomposition of $T$ by the nine-square representation from Theorem~\ref{thm:T9} gives a 17-square representation. Hence both 18-square candidates are reducible, so $z_2(5,5)\le17$. Since $\zRL(5,5)=17$ is known, the hierarchy \eqref{eq:hierarchy} gives equality.
\end{proof}

\subsection{A general method suggested by the $5\times5$ case}

The $5\times5$ proof provides a second proof pattern, complementary to the four-column chain:
\[
\begin{aligned}
\text{ordinary-skeleton classification}
&\longrightarrow \text{finite compatibility reduction}\\
&\longrightarrow \text{highly symmetric residual cases}\\
&\longrightarrow \text{low-rank Gram / explicit shorter SOS}.
\end{aligned}
\]
The most immediate general consequence is the following.

\begin{proposition}[Global inheritance of a local shorter SOS]\label{prop:localglobal}
Suppose a complete displayed sub-sum
\[
F=f_1^2+\cdots+f_s^2
\]
has a shorter SOS representation
\[
F=g_1^2+\cdots+g_r^2,
\qquad r<s,
\]
Then every larger displayed decomposition containing these complete displayed squares is reducible. If there are $q$ such sub-sums that share no displayed square and each saves at least one square, then the total SOS length can be reduced by at least $q$.
\end{proposition}
\begin{proof}
Keep all other squares unchanged and replace these sub-sums one at a time.
\end{proof}

Thus \eqref{eq:T} gives a local reducible pattern valid for every $m,n\ge4$: whenever a larger configuration contains this complete ten-square sub-sum on four distinct rows and four distinct columns, the entire displayed decomposition is reducible.

Proposition~\ref{prop:Trank} shows that the ten original displayed forms of $T$ admit no nonzero product relation capable of producing a rank reduction within their span, yet \eqref{eq:T9} gives a nine-square representation. Thus the absence of a Gram perturbation in the original displayed space is not, in general, a criterion for irreducibility: a shorter representation may use bilinear forms outside that span. This explains why, in the seven- and eight-row exclusions, linear algebra restricted to the original displayed space is insufficient; when necessary, one must search directly for low-rank positive semidefinite Gram representations in the full cell-monomial space.

\subsection{Further questions and conjectures}

\begin{conjecture}[L\"ofberg--Qi]\label{conj:z2zrl}
For all $m,n$,
\[
z_2(m,n)=\zRL(m,n).
\]
\end{conjecture}
The new exact values at $(4,4)$, $(7,4)$, $(8,4)$, and $(5,5)$ provide additional finite evidence for this conjecture, but do not imply the general statement.

\begin{problem}
Theorem~\ref{thm:T9} proves only $\SOS(T)\le9$. Is it true that
\[
\SOS(T)=9\,?
\]
\end{problem}

\begin{problem}
For fixed $n\ge4$, does there exist a finite set of local reducible patterns depending only on $n$ such that every irreducible limited configuration sufficiently close to the cell bound must avoid all of them?
\end{problem}

Previous work also raises an eventual-saturation question at fixed width: for every fixed $n\ge4$, does there exist a threshold $g_2(n)$ such that for all $m\ge g_2(n)$
\[
z_2(m,n)=\left\lfloor\frac{mn+z(m,n)}2\right\rfloor.
\]
For four columns this is known for $m\ge15$ \cite{CC}, whereas the three-column family does not attain the corresponding cell bound \cite{LQ,QCL}. The ten-to-nine-square compression in the $5\times5$ case suggests that, at higher widths, local Gram-rank reductions may influence the eventual behavior in addition to the combinatorics of the ordinary skeleton.

\section{Conclusion}

Starting from upper bounds for the second-order Zarankiewicz number, we studied the distinction between the parameters $\zRL$ and $\zSL$, defined through particular sufficient conditions, and the unrestricted irreducible extremum $z_2$. The guiding principle throughout is that positive cases may be certified by sufficient conditions such as $(RW3^+)$, whereas negative cases require explicit reducibility proofs; failure of a sufficient condition is never used as a counterargument.

For $4\times4$, we first determine the unique ordinary extremal skeleton and then use two-row/two-column strip bounds together with one product relation to show that an irreducible configuration contains at most one two-edge, giving $z_2(4,4)=10$. In larger four-column sizes, the ordinary extremal skeleton stabilizes from six rows onward. We classify all extremal irreducible $6\times4$ configurations with 16 squares and obtain a unique isomorphism class; deleting the new degree-one row then excludes 20-square candidates on seven rows and yields $z_2(7,4)=19$. To continue to eight rows, we completely classify the 19-square seven-row extremizers into three isomorphism classes. These three templates then constrain the 22-square eight-row candidates, leading to $z_2(8,4)=21$. This develops the structural method ``small exact value and extremal classification -- deletion restriction -- next-order exclusion and yields the general deletion inequality \eqref{eq:generaldelete}.

In the $5\times5$ case there are two ordinary extremal skeletons. Finite enumeration, together with local strip and product-relation exclusions, leaves only two highly symmetric families. They define the same ten-square polynomial $T$, which has an explicit nine-square representation, and therefore $z_2(5,5)=17$. This proof supplies a second method with broader potential: first reduce the candidate set combinatorially, then search for a low-rank positive semidefinite Gram representation of the highly symmetric residual cases. At the same time, every known local shorter SOS becomes automatically a reducible substructure in all larger dimensions.

These results do not settle the general conjecture $z_2=\zRL$, nor do they complete all small four-column extremal classifications. They indicate that further progress will likely require three levels simultaneously: extremal structure of the ordinary $C_4$-free skeleton, isomorphism classification of irreducible augmented configurations, and Gram-rank reductions that leave the original displayed space. For $9\times4$ and the other unresolved four-column parameters, a natural next step is to classify all 21-square extremal $8\times4$ configurations and combine that classification with the deletion inequalities and local shorter-SOS patterns.

\section*{Data and computational verification}

Every computer-assisted component of the paper is finite. Computation is used only to execute the explicit steps of Algorithm~\ref{alg:finite} and the independent verification in Algorithm~\ref{alg:verify}: enumerate pairings on a fixed finite grid, check local strip and product-relation reductions, build the necessary compatibility graph, enumerate cliques, compute automorphism orbits of the ordinary skeleton, and verify the polynomial identities or Gram matrix identities described in Section~\ref{subsec:certs}. The numerical tables in the text give the finite ranges that must be exhausted, the orbit counts after pruning, and the counts for each proof type. A candidate is declared reducible only when a strip bound, product relation, explicit shorter SOS, or strictly lower-rank positive semidefinite Gram representation is present. Failure of $(RW3^+)$ is never used as a negative proof.

Discovery and verification are kept separate. Numerical optimization may suggest a low-rank Gram matrix or rewrite coefficients, but the archived records contain only exact integers, rationals, or elements of $\mathbb Q(\sqrt2)$. The independent checker reconstructs candidate indices from all cell pairs, expands every local exclusion, independently enumerates all target cliques of the compatibility graph, and builds the orbits under the full skeleton automorphism group. It verifies the explicit reducibility proof or exact Gram representation for every negative orbit and replays the $(RW3^+)$ closure of Definition~\ref{def:rw3rules} for every positive orbit. The machine-readable certificates and the independent checker are supplied as Supplementary Data~$\mathrm{S1}$, containing at least the following independently verifiable items.
\begin{enumerate}[leftmargin=2em,itemsep=0.35em]
\item The 10-square $4\times4$ positive example, the unique 16-square extremal representative on $6\times4$, the three 19-square representatives on $7\times4$, and the 21-square lower-bound representative on $8\times4$, each with a complete replay of the $(RW3^+)$ closure under Definition~\ref{def:rw3rules}.
\item The $7\times4$ 19-square finite classification: after strip pruning, 994 labeled families form 170 automorphism orbits; 86 orbits are excluded directly by nonzero product relations in the original displayed space, while the remaining 84 consist of 44 explicit shorter integer SOS cases, 10 equal-length rewrites followed by product relations, 2 low-rank Gram matrices over $\mathbb Q(\sqrt2)$, 25 low-rank Gram matrices from rational contraction arguments, and 3 $(RW3^+)$-admissible positive orbits.
\item The $8\times4$ 22-square exclusion: for ordinary skeleton $\mathcal A_8$, 1202 orbits after strip pruning are excluded directly by product relations in the original displayed space, and the remaining 829 are excluded by 283 seven-row restrictions, 434 explicit shorter integer SOS representations, 71 equal-length rewrites followed by product relations, and 41 low-rank positive semidefinite Gram representations. For $\mathcal B_8$, the corresponding numbers are 4607 direct product-relation orbits and 3758 remaining orbits, split as 1365 seven-row restrictions, 1880 shorter integer SOS representations, 265 equal-length rewrites followed by product relations, and 248 low-rank positive semidefinite Gram representations.
\item The $5\times5$ finite exclusion record: the 78 single two-edge candidates for each ordinary skeleton, all local two-candidate exclusions, the necessary compatibility graphs and all of their 6-cliques, together with the two final 18-square exceptional families $\mathcal F_+$ and $\mathcal F_-$, the ten-square displayed decomposition of $T$, and the coefficients of the nine-square representation \eqref{eq:T9}.
\item The independent checker: reconstructs candidate indices from cell pairs, replays local exclusions, compatibility graphs, target cliques, and automorphism orbits, and verifies each reducibility proof or $(RW3^+)$ closure above.
\end{enumerate}

\end{document}